\documentclass[12pt]{amsart}

\usepackage{caption,amsmath,amssymb, enumerate,
xcolor, url,  paralist}

\usepackage[top=1.35in, bottom=1.35in, left=1.4in,
right=1.4in]{geometry}

\def\E{\mathbb E}
\def\P{\mathbb P}
\def\R{\mathbb R}
\def\N{\mathbb N}
\def\Bc{\mathcal B}
\def\Fc{\mathcal F}
\def\Hc{\mathcal H}
\def\Lc{\mathcal L}
\def\Mc{\mathcal M}
\def\Nc{\mathcal N}

\newtheorem{theorem}{Theorem}[section]
\newtheorem{proposition}[theorem]{Proposition}
\newtheorem{lemma}[theorem]{Lemma}

\title[Buffer-Hawkes Process]{A Buffer Hawkes Process for Limit Order Books}

\author{Ingemar Kaj
  \and  Mine Caglar}
\thanks{I. Kaj, Dept.\ of Mathematics, Uppsala University, Uppsala,
    Sweden ({\tt ikaj@math.uu.se})} 
 \thanks{M. Caglar, Dept.\ of Mathematics, Koc University, Istanbul, Turkey
     ({\tt mcaglar@ku.edu.tr})}

\date{October 10, 2017}

\begin{document}


\begin{abstract}
  We introduce a Markovian single point process model, with random
  intensity regulated through a buffer mechanism and a self-exciting
  effect controlling the arrival stream to the buffer. The model
  applies the principle of the Hawkes process in which point process
  jumps generate a shot-noise intensity field. Unlike the Hawkes
  case, the intensity field is fed into a separate buffer, the size of
  which is the driving intensity of new jumps. In this manner, the
  intensity loop portrays mutual-excitation of point process events
  and buffer size dynamics.  This scenario is directly applicable to
  the market evolution of limit order books, with buffer size being
  the current number of limit orders and the jumps representing the
  execution of market orders. We give a branching process
  representation of the point process and prove that the scaling limit
  is Brownian motion with explicit volatility.
\end{abstract}

\keywords{ 
Self-exciting, market price, order book, branching process}

\subjclass[2010]{60G55, 91G80, 60J80}

\maketitle

\section{Introduction}

The self-exciting point process known as the Hawkes process is a
highlighted model of choice in a number of recent mathematical studies
of limit order books (LOB) \cite{abergelbook2016}. It represents the
dependence of the interarrival times of market events in order to
match the empirical order flow and spread dynamics \cite{lobsurvey}.
The point of view behind the approach is that the execution events of
bids and calls when orders are removed from the book trigger an
increase of the rate at which new orders enter the limit order book.
In this paper, we construct and analyze a more general self-exciting
process which aims at targeting more directly the basic dynamics of
limit orders and market orders. It involves the Markovian Hawkes
process in part, and captures mutual-excitation observed in limit
order books in a nonstandard way compared with multidimensional Hawkes
processes.

A large and growing literature is devoted to application of Hawkes
processes in finance \cite{survey}. The Markovian Hawkes process is a
process $(\Lambda_t,N_t)$, such that $N_t$ is a counting process and
$\Lambda_t$ has exponentially decreasing paths and positive jumps
generated by $N_t$ according to
\begin{equation}\label{LambdaNrelation}
\Lambda_t=\lambda_0+\int_0^t ae^{-b(t-s)}\,dN_s, \quad t\ge 0,
\end{equation}
with $\lambda_0>0$ and $a,b$, $0<a<b$, parameters, and furthermore,
$\Lambda_t$ is the stochastic intensity of $N_t$.  The cluster
representation due to Hawkes and Oakes \cite{hawkesoakes1974} reveals
the close link between self-exciting processes and branching models.
In this view, $\lambda_0$ is the intensity of arriving immigrants and
$\xi(ds)=ae^{-b s}\,ds$ is the intensity of a Poisson offspring
distribution on the positive half line of a subcritical branching
process with mean offspring $\nu_0=\int_{\R_+} \xi(ds)=a/b<1$, which
counts additional events due to the internal feedback. The Hawkes
process is the aggregate overlay of independent branching processes
initiated by each immigrant at the time of arrival.

In finance and other applications, processes are studied
in greater generality than the exponential shot noise in
(\ref{LambdaNrelation}). Intensity processes of the form
\begin{equation} \label{hawkes}
\Lambda_t = \Lambda_0 +\int_0^t \phi(t-s)\,
dN_s=\Lambda_0+\sum_{s_i\le t} \phi(t-s_i),
\end{equation}
where $\phi$ is a suitable nonnegative function and $\Lambda_0$ is
random or constant, defines a wider class of linear Hawkes models with
self-excited event times $s_1,s_2,\dots$, \cite{bremaudmass1996}.
Order book modeling and financial data analysis as a rule apply
multivariate point processes $\{(s_i,d_i),i\ge 1\}$, where $d_i$ is
the component index of a vector valued counting process
$N_t=(N^{1}_t,\dots,N_t^{n})$ with intensity
$\Lambda_t=(\Lambda_t^1,\dots,\Lambda_t^n)$ such that
\[
\Lambda_t^j=\Lambda_0^j+\sum_{i=1}^n\int_0^t
\phi_{ij}(t-s)\,dN_s^i,\quad 1\le j\le n.
\]
In LOB context, typically, the Hawkes point process $N^j$ would count
limit order arrivals or market order executions of asset $j$ while the
intensities $\Lambda^j$ control the rate of those orders. An
enlargement of this model involves another set of Hawkes processes for
the limit orders of assets $j=1,\ldots ,n$ and allows for mutual
excitation between different types of orders as well, see
e.g. \cite{abergelbook2016}.  As for the long time behavior of the
univariate model the strong law of large numbers is $N_t/t\to
(1-\nu_0)^{-1}$, almost surely, as $t\to \infty$ and the functional
central limit theorem is the weak convergence
\[
\frac{1}{\sqrt{m}}\Big(N_{mt}-\frac{1}{1-\nu_0} mt\Big)
\Rightarrow \sigma\, B_t,\quad m\to\infty,\quad
\sigma^2=\frac{1}{(1-\nu_0)^3},
\]
where $B$ is Brownian motion. Analogous results for multivariate
Hawkes models are established in \cite{bacryetal2013} under the
assumption that the matrix $(\Phi_{ij})$, $\Phi_{ij}=\int_0^\infty
\phi_{ij}(t)\,dt$, has spectral radius strictly less than 1.

An important feature captured by Hawkes processes is clustering in
time, which is a well observed empirical fact for order arrivals
\cite{quantfinance2011, Zaatour}. This is due to the self-exciting
property, that is, the jump intensity increases with the process
itself.  Another feature of order flows as demonstrated by several
studies is mutual excitation. Market orders excite limit orders, and
limit orders that change the price excite market orders
\cite{eisleretal2012,AberJedi2015}.  Aiming at a mechanism for reproducing such
stylized facts, we propose an intertwined model of
\begin{compactitem}
\smallskip
\item market order arrivals $N$, with intensity proportional to existing limit orders $\Gamma$;
\item limit order arrivals $L$, with intensity $\Lambda$ excited by market orders;
\item  cancellation of an arriving limit order,
\smallskip
\end{compactitem}
in the LOB for a single asset. We show that all these events come together in a coherent and analytically tractable way. The model represents the contagion of the limit orders and the market orders from each other.

Explicitly, our approach is to consider a Markov process
$$X_t=(\Lambda_t,\Gamma_t,N_t)$$
where $\Lambda$ and $N$ are still linked by relation
(\ref{LambdaNrelation}), but where $\Lambda$ is no longer the
stochastic intensity of the \emph{market order} process $N$.  In
contrast, $\Lambda$ is the stochastic intensity of an auxiliary
process $L_t$ which is the arrival process of an LOB infinite server
buffer for \emph{limit orders}. The arrival events are placements of
new limit orders while departures from the service system are limit
order cancellations or market order executions, and it is (a multiple
of) the resulting buffer size $\Gamma$ which is now the stochastic
intensity of $N$.  The interpretation of $X$ is that $\Lambda$ is the
intensity process for limit orders entering the LOB, $\Gamma$ is the
current size of the LOB, and $N$ counts the accumulated number of
executed market orders.  In this model, the number of entries in the
LOB arises as the net balance of limit orders either arriving,
triggered by market order executions, or departing, as the result of
cancellation or execution, and therefore we refer to $X$ as a
\emph{buffer-Hawkes process}.
 We derive explicit formulas for the
first and second order properties of buffer-Hawkes process.
 We demonstrate its branching process representation, which reveals the clustering caused by mutual excitation, and in turn, the dependence of the increments.
The diffusion limit is found as a result of long time scaling of the process.

We consider the  market impact through a mid-price  model
\[
S_t=S_0+\left(N^{+}_t-N^{-}_t\right)\, \frac{\alpha}{2},
\]
where $N^+$ and $N^-$ are market price counts, $+$ stands for up, that
is buy, and $-$ stands for down, that is sell, $\alpha>0$ is a tick
price parameter and $S$ is the price of the asset, as in
\cite{Zaatour} where this set-up is called the toy-model.  We obtain
the scaling limit of $S$ on the basis of those results for each market
order assuming $N^+$ and $N^-$ are independent. Martingale machinery
is crucial for obtaining tightness in the functional central limit
theorem even in the independent case. The asymptotic volatility is
obtained explicitly as a function of the model parameters.  A simple
model is useful for example in analyzing the optimal high-frequency
trading strategy. A similar market-impact model is proposed for
obtaining a dynamic optimal execution framework in
\cite{alfonsiblanc2016} where the price expression includes a multiple
of the market orders that obey a Hawkes process. The market orders
affect the drift of the mid-price in \cite{carteaetal2014}, which aims
at a control strategy on the basis of self-exciting buy and sell
orders.

Related work on Hawkes processes and limit order books shed light on
future work with the model of the present paper.  Multidimensional
processes to include several assets have been considered and long-time
behavior is established in e.g. \cite{AberJedi2015}.  The queues at
different tick levels on the ask and bid sides of a limit order book
are considered in \cite{contlarrard2013}. For theoretical study of
more general Hawkes processes, a functional central limit theorem is
proven in the context of queues in \cite{GaoZhu}. In
\cite{zhengetal2014}, bid and ask prices are studied using Hawkes
processes with additional constraints. In addition to endogenous
effects modeled by Hawkes mechanism, exogenous effects are also
appended to the intensity process in \cite{contag} to model the
contagion impact from various factors of the underlying system. For a
vast list of further references, we refer to
\cite{survey,bacryetal2013}.  Most recently, in \cite{jiaomascotti2017} an
integral process is shown to arise as the limit of Hawkes processes
and to exhibit, by their nature, the self-exciting property.

Organization of the paper is as follows. In Section 2, we motivate a
preliminary version and then construct the buffer-Hawkes process for
limit order books. We also compute the first and second order
moments. In Section 3, the branching representation of the process is
given. The diffusion scaling as a long time behavior is studied in
Section 4, where both the market order process and the price process
are considered.

\section{Model of Market and Limit Orders}
\label{sec:model}

We develop a model for the limit order book step by step considering
the events that excite one another. First, the market arrival process
$N$ is constructed as a self-exciting process by a process
$\Gamma$. Then, the latter is linked to the limit orders so that
mutual excitation between market and limit orders is captured.

Our approach follows the construction in \cite[Thm.VI.6.11]{cinlar2011}.  We
begin with a probability space $(\Omega,\Hc,\P)$ and the measurable
space $(\R_+\times\R_+,\Bc)$ equipped with a filtration $\Fc=(\Fc_t)$
adapted to the Borel product $\sigma$-algebra $\Bc=(\Bc_{\R_+}\otimes
\Bc_{\R_+})$.  Let $M: \Omega\times \Bc\mapsto [0,+\infty)$ be a
Poisson random measure on $(\R_+\times\R_+,\Bc)$ with Lebesgue
intensity measure $\mu(ds,dz)=ds\,dz$ relative to $\Fc$.  This means
that $\omega\mapsto M(\omega,B)=M(B)$ is a random variable for each
$B\in \Bc$ and $B\mapsto M(\omega,B)$ is a measure on
$(\R_+\times\R_+,\Bc)$ for each $\omega$.  Moreover, for each
$B\in\Bc$ the random variable $M(B)$ is Poisson distributed with
expectation $\mu(B)=\int_B \,\mu(ds,dz)$ and for $B_1,\dots,B_n\in
\Bc$, $n\ge 2$, the variables $M(B_1),\dots,M(B_n)$ are
independent. Furthermore, $M(B)\in \Fc_t$ for every $B\in \Bc_{[
  0,t]}\otimes \Bc_{\R_+}$, and the measure $M$ restricted to the
set $(t,\infty)\times \R_+$ is independent of $\Fc_t$.

Let $\Gamma=(\Gamma_t)_{t\ge 0}$ be a nonnegative process with
c\`adl\`ag paths adapted to $\Fc$ and let $(\Gamma_{t-})_{t\ge 0}$ be
the left-continuous and hence $\Fc$-predictable version of $\Gamma$.
Suppose that $N$ is a counting process that satisfies
\[
N_t = \int_0^t \int_0^\infty\, 1_{(0,\Gamma_{s-}]}(z)\, M(ds,dz).
\]
Then, the process
\[
N_t -\int_0^t \Gamma_s \,ds = \int_0^t \int_0^\infty 1_{(0,\Gamma_{s-}]}(z)\, (M(ds,dz)-ds\, dz)
\]
is an $\Fc$-martingale. Its quadratic variation is given by
\[
\int_0^t \int_0^\infty
1_{(0,\Gamma_{s-}]}(z)\,dsdz = \int_0^t \Gamma_s \, ds
\]
as $1^2=1$. Here, $(\Gamma_{t-})$ is the stochastic
intensity process for $N$.  To add flexibility to the
model it is convenient to include a further parameter $c>0$
controlling the rate of impact of $\Gamma$ on
$N$. To do so, replace $M(ds,dz)$ by a Poisson random measure $M_c(ds,dz)$
with intensity measure $\mu_c(ds,dz)=c\,ds\,dz$, or replace
$\Gamma$ in the definition of $N$, by $c\Gamma$, so that
\begin{equation}\label{Npoissondef}
N_t = \int_0^t \int_0^\infty 1_{(0,c\Gamma_{s-}]}(z)\, M(ds,dz)\, .
\end{equation}
With this extension the compensated process $N_t-c\int_0^t
\Gamma_s\,ds$, $t\ge 0$, is an $\Fc$-martingale, and $c$ controls the
strength of the feedback mechanism by which jumps in $N$ influence the
intensity of additional jumps later on.  Next, introduce $\Lambda_t$
to be the shot noise process generated by $N$ for a given shot profile
function $\phi$, as defined in (\ref{hawkes}).

The Hawkes process arises by the choice $\Gamma=\Lambda$, hence
postulating that the intensity process of $N$ is the predictable
version $(\Lambda_{t-})_{t\ge 0}$ of the self-referential process
$\Lambda$. The setting in (\ref{hawkes}) yields a general class of
non-Markovian Hawkes processes whereas the case $\Lambda_0=\lambda_0$
and $\phi(t)=ae^{-bt}$ is the classical Hawkes process
(cf. \cite[pg.311]{cinlar2011}, \cite{survey}).  In the latter case
$\Lambda_t$ is the base level-reverting shot-noise process with
exponential pulse function defined in (\ref{LambdaNrelation}), such
that $d\Lambda_t=-b(\Lambda_t-\lambda_0)\,dt+a\,dN_t$ and
$\Lambda_t\ge \lambda_0$.  In terms of the underlying Poisson random
measure $M_c$, these relations show further that $\Lambda_t$ is a
solution of the stochastic integral equation
\[
\Lambda_t=\lambda_0+a\int_0^t\int_0^{\Lambda_{s-}} e^{-b(t-s)}\, M_c(ds,du),
\]
as discussed in \cite{jiaomascotti2017}.  For this classical case with
fixed $a$ and $b$, and taking into account the parameter $c$, it is
well-known that $\Lambda_t$ has a stationary distribution for
$ac<b$. Under this assumption, by taking expectations in the previous
displayed relation, or relation (\ref{LambdaNrelation}),
\[
E(\Lambda_t)=\lambda_0+\int_0^t ae^{-b(t-s)}\,c\,
E(\Lambda_s)\,ds,
\]
and hence $E\Lambda_\infty=\lambda_0 b/(b-ac)$.

The infinitesimal generator $\Lc
f(\lambda,n)=\frac{d}{dt}\E[f(\Lambda_t,N_t)|\Lambda_0=\lambda,N_0=n]_{t=0}$,
of the continuous time Markov process $(\Lambda_t,N_t)$ satisfies for suitable
functions $f$
\begin{eqnarray}\nonumber
\Lc f(\lambda,n)=b(\lambda_0-\lambda)\frac{\partial f}{\partial \lambda}(\lambda,n)
+c\gamma(f(\lambda+a,n+1)-f(\lambda,n)).
\label{}
\end{eqnarray}
The cluster representation is used systematically in the present
work. We recall the basic case here in a setting that will be extended
to $X_t$.  The parameters $\lambda_0>0$ and $a,b$, $0<a<b$ are
fixed. Let $\Nc(ds)$ and $\Mc(ds)$ be independent Poisson random
measures on the positive real line with intensity measures
$n(ds)=\lambda_0\,ds$ for $\Nc$ and $\xi(ds)=ae^{-b s}\,ds$ for
$\Mc$. Let $Z_t$ be the subcritical branching process with Poisson
offspring intensity $m(ds)$ and mean offspring $\nu_0=\int_{\R_+}
\xi(ds)=a/b<1$, which satisfies the branching relation
\[
Z_t=1+\int_{\R_+} 1_{\{s\le t\}} Z_{t-s}^{(s)}\,\Mc(ds)
\]
where $\{Z^{(s)}\}$ are independent copies of $Z_t$.  The Hawkes
process is the aggregate of independent branching processes which is
generated by a sequence of immigrants that arrive according to $\Nc$,
that is
\begin{equation}\label{branchingclassical}
N_t=\int_{\R_+}1_{\{s\le t\}}Z_{t-s}^{(s)}\,\Nc(ds).
\end{equation}

\subsection{Buffer-regulated basic counting process}
\label{sec:2.1}

In this section, we discuss buffer mechanisms driving the intensity
process of $N$.  For the simplest instance, let $\Lambda_0=\lambda_0$
be constant and take $a=0$ so that the effect of the Hawkes mechanism
is turned off and the intensity process is trivial,
$\Lambda_t=\lambda_0$. Still, $\lambda_0$ determines the Poisson rate
of arriving limit orders. Let $L_t$ be the corresponding Poisson
process with intensity $\lambda_0$ and let $N_t$ be defined by
(\ref{Npoissondef}) with
\[
\Gamma_t = (L_t- N_t)^+,
\]
which is the buffer size of an M/M/$\infty$ queue with
parameters $\lambda$ and $c$ \cite[Exer.VI.6.53]{cinlar2011}. Moreover, since
$\Gamma_t=0$ when $L_{t-}\leq N_{t-}$, we have $N_t \leq L_t$ for all
$t$. This fact and $M$ and $L$ being independent from each other help
to show that $(\Gamma_t,N_t)$ is Markov. By the same arguments
$\Gamma_t$ is also Markov.
The resulting counting process $N$ is self-regulating in the
sense that if $N$ happens to have many jumps in a given time interval then
its intensity process is reduced accordingly, causing the further accumulation
of jumps to slow down.

To make this example more relevant for the LOB context we add
cancellations. Suppose each limit order is cancelled at a constant
rate $d\ge 0$. Put
\[
K_t= \mbox{number of non-cancelled limit orders arrived by $t$},
\]
so that $K$ is an M/M/$\infty$ buffer process of limit
orders with arrival rate $\lambda_0$ and cancellation rate $d$.
Again let $N$ be the counting process with intensity
$(c\,\Gamma_{t-})$, but now let $\Gamma$ be the dynamic storage
process
\begin{equation}\label{def:buffer}
\Gamma_t = \sup_{r\le t}(K_t-K_r- (N_t-N_r))
         = K_t-N_t+\sup_{r\le t}\{-(K_r-N_r)\},\quad t\ge 0.
\end{equation}
The continuous time Markov process $(\Gamma_t,N_t)$
has infinitesimal generator $\Lc f(\gamma,n)$, given by
\begin{align}\nonumber
\Lc f(\gamma,n)&=\lambda_0(f(\gamma+1,n)-f(\gamma,n)) + d\gamma ( f(\gamma-1,n)-f(\gamma,n))  \\
& \quad    +c\gamma(f(\gamma-1,n+1)-f(\gamma,n)).
\label{infgen_a=0}
\end{align}
The counting process $N$ causes $\Gamma$ to decrease with rate $c$ as
well, so marginally $\Gamma$ is distributed as the $M/M/\infty$ buffer
with parameters $\lambda_0$ and $c+d$ as evident from the generator
$\Lc$.

\subsection{Buffer-Hawkes process}
\label{sec:2.2}

We are now prepared to construct what we call a buffer-Hawkes process
by combining the self-exciting Hawkes mechanism with the
self-regulating buffer mechanism.  Let $S=\R_+\times \N\times \N$. The
buffer-Hawkes process is a Markov process
$X_t=(\Lambda_t,\Gamma_t,N_t)$ with c\`adl\`ag realizations on the
state space $S$, with $\Lambda_t$ standing for the intensity of limit
orders entering the limit order book, $\Gamma_t$ standing for the size
of the limit order book, and $N_t$ standing for the number of executed
market orders.  The set of parameters in the model are $\lambda_0>0$,
$a\ge 0$, $b>0$, $c>0$, $d\ge 0$ for which we assume the stability
condition
\begin{equation}\label{stabilitycondition}
ac<b(c+d),
\end{equation}
which is imposed from now on.

We assume that $\Gamma$ and $N$ are
related by (\ref{Npoissondef}), that is,
\[
N_t = \int_0^t \int_0^\infty 1_{(0,c\Gamma_{s-}]}(z)\, M(ds,dz)
\]
so that $N_t$ is a pure jump process
with compensator $c\int_0^t \Gamma_s\,ds$, and that $\Lambda_t$ is
obtained from the jumps of $N_t$ as defined in
(\ref{LambdaNrelation}), that is,
\[
\Lambda_t=\lambda_0+\int_0^t ae^{-b(t-s)}\,dN_s\; .
\]
  Let $\Xi(ds,du)$ be a Cox random measure
on $\R_+\times \R_+$ with stochastic intensity measure
$\Lambda_{t-}\,dt\,d\, e^{-d\,u}\, du$, that is,  a conditionally Poisson random measure given $\Lambda$ \cite[pg.262]{cinlar2011}. Then
\[
L_t=\int_{\R_+^2} 1_{\{s\le t\}} \,\Xi(ds,du)
\]
is the pure jump process with compensator $\int_0^t\Lambda_s\,ds$ and
the nonnegative integer-valued process
\[
K_t=\int_{\R_+^2} 1_{\{s\le t\le s+u\}} \,\Xi(ds,du)
\]
is the corresponding $L_t/M/\infty$ buffer process with arrival
process $L_t$ and exponential cancellation times of rate $d$. The
special choice $a=0$ is the previously studied case when $L_t$ is a
Poisson process with intensity $\lambda_0$ and $K_t$ is the
M/M/$\infty$ process with parameters $\lambda_0$ and $d$.
Finally, to complete the construction of $X$, let $\Gamma_t$ be the buffer
process which results from the net input $K_t-N_t$, as defined in
(\ref{def:buffer}), that is,
\[
\Gamma_t = \sup_{r\le t}(K_t-K_r- (N_t-N_r))
,\quad t\ge 0.
\]
Note that $\Lambda_t$ and $\Gamma_t$ correspond to the arrival
intensity of the limit and market orders, respectively, in the limit
order book.

The construction of $X_t$ implies that $\Gamma_t$ is a birth-death
process with births given by $L_t$ and deaths shared with the downward
jumps of $K_t$ and the upward jumps of $N_t$, as long as the buffer is
non-empty.  The simultaneous jumps structure is the key to writing
down the Markov generator for $X$. For $x=(\lambda,\gamma,n)\in S$ and
functions $f$ on $S$ we denote $\E_x f(X_t)=\E(f(X_t)|X_0=x)$.  The
generator $\Lc$ of $X$ defined by $\Lc f(x)=\frac{d}{dt}E_x
f(X_t)|_{t=0}$, for sufficiently regular $f$, has the form
\begin{align}\nonumber
\Lc f(x)&=b(\lambda_0-\lambda)\frac{\partial f}{\partial
  \lambda}(x)+\lambda(f(\lambda,\gamma+1,n)-f(x))\\
&\quad + d \gamma ( f(\lambda,\gamma-1,n)-f(x))
+c\gamma(f(\lambda+a,\gamma-1,n+1)-f(x)).
\label{generator}
\end{align}
Consequently, for such $f$,
\[
\widetilde X_t[f]=f(X_t)-f(x)-\int_0^t \Lc f(X_s)\,ds, \quad t\ge 0,
\]
is a zero-mean $\Fc$-martingale.  In particular, using $f_1(x)=\lambda$,
$f_2(x)=\gamma$, $f_3(x)=n$, this yields the semimartingale
representations
\begin{align*}
\Lambda_t&=\lambda_0+\int_0^t (b(\lambda_0-\Lambda_s)+ac\Gamma_s)\,ds
+\widetilde X_t[f_1]\\
\Gamma_t&=\int_0^t (\Lambda_s-(c+d)\Gamma_s)\,ds+\widetilde X_t[f_2] \quad
\\
N_t&=\int_0^t c\Gamma_s\,ds+\widetilde X_t[f_3].
\end{align*}
To have these relations consistent with the existence of finite
expected values in equilibrium as $t\to\infty$, one needs
$b\lambda_0-b\E(\Lambda_\infty)+ac\E(\Gamma_\infty)=0$ and
$\E(\Lambda_\infty)=(c+d)\E(\Gamma_\infty)$. Hence
$\E(\Gamma_\infty)=\lambda_0 b/(b(c+d)-ac)<\infty$, due to the
stability condition (\ref{stabilitycondition}).

To close this subsection, we comment on the two-dimensional marginals
of $X$.
The marginal distribution $(\Lambda_t,N_t)$ is not Markovian, neither
is $(\Gamma_t,N_t)$ except for the case $a=0$ in (\ref{infgen_a=0}).
The marginal distribution $(\Lambda_t,\Gamma_t)$ is a Markov process
with generator
\begin{align*}
\Lc f(\lambda,\gamma)&=b(\lambda_0-\lambda)\frac{\partial f}{\partial
  \lambda}(x)+\lambda(f(\lambda,\gamma+1)-f(\lambda,\gamma))\\
&\quad + d \gamma ( f(\lambda,\gamma-1)-f(\lambda,\gamma))
+c\gamma(f(\lambda+a,\gamma-1)-f(\lambda,\gamma)).
\end{align*}
This bivariate process is a self-exciting modification of the
M/M/$\infty$ model where each departure from the service system
independently with probability $c/(c+d)$ triggers a novel shot-noise
contribution to the intensity of subsequent customer arrivals.

\subsection{First and second order moments}

We derive expressions for the first and second order moments of the
buffer-Hawkes process in terms of the parameters $a,b,c,d$ and
$\lambda_0$. Put
\[
Q=\sqrt{(b-c-d)^2+4ac},\quad q_-=(b+c+d-Q)/2,\quad q_+=(b+c+d+Q)/2.
\]

\begin{proposition} Denote $\E(X_t)=(\ell_t,g_t,m_t)$. We have
\begin{align*}
\ell_t
&=\frac{\lambda_0}{q_+-q_-}
   \Big(q_+-q_-+ac \,\int_0^t (e^{-q_-s}-e^{-q_+s})\,ds\Big)\\
g_t&=\frac{\lambda_0}{q_+-q_-}\Big(e^{-q_-t}-e^{-q_+t}+b\int_0^t
(e^{-q_-s}-e^{-q_+s})\,ds\Big)\\
m_t&=c\int_0^t g_s\,ds.
 \end{align*}
\end{proposition}

\begin{proof}
Using Dynkin's formula for the generator introduced in
(\ref{generator}),
\begin{equation}\label{dynkinformula}
\E_x f(X_t)=f(x)+\E_x\int_0^t \Lc f(X_s)\,ds,\quad X_0=x,
\end{equation}
it is straightforward to derive coupled systems of ODE's for the first
and second order moments, cf. \cite[Chp.2]{Zaatour}.
 Then, using (\ref{dynkinformula}),
\begin{align*}
\left[\begin{array}{c}
           \ell_t'\\
           g_t'\\
           m_t'
\end{array}\right]
+ \left[\begin{array}{ccc}
         b & -ac & 0  \\
         -1 & c+d & 0  \\
          0 & -c & 0
\end{array}\right]
\left[\begin{array}{c}
           \ell_t\\
           g_t\\
           m_t
\end{array}\right]
=\left[\begin{array}{c}
           b\lambda_0\\
           0\\
           0
\end{array}\right]
\end{align*}
In view of (\ref{stabilitycondition}),
\[
q_+\ge q_-> 0,\quad q_+-q_-=Q\ge 0,\quad q_-q_+=b(c+d)-ac>0.
\]
For $ac>0$ we have $Q>0$. The solutions of the linear ODE system yield the result.
\end{proof}

Asymptotically as $t\to\infty$, we get
\[
\ell_t\to \ell_\infty=\frac {b(c+d)\lambda_0}{b(c+d)-ac},\quad
g_t\to g_\infty=\frac{\ell_\infty}{c+d},\quad m_t\sim \frac{c\,\ell_\infty}{c+d}\,t
\]

Note that the case $a=0$ with $c>0$ is the basic buffer model in
Section \ref{sec:2.1}, for which $q_-=(c+d)\wedge b$, $q_+=(c+d)\vee
b$, and $Q=|b-c-d|$. Then, for any $b$,
\[
\ell_t=  \lambda_0 ,\quad
g_t=\frac{\lambda_0}{c+d}
(1-e^{-(c+d)t} )\to
\frac{\lambda_0}{c+d}, \quad
m_t=\frac{c\lambda_0}{c+d}\int_0^t (1-e^{-(c+d)s})\,ds
   \sim \frac{c\lambda_0}{c+d}\,t.
\]


We compute all second order moments in pursuit of asymptotic variance of $N_t$ as given next.

\begin{proposition} For large $t$, we have
\begin{align*}
V(N_t)\sim
   &\frac{\lambda_0 c}{c+d}\Big(\frac{b(c+d)}{b(c+d)-ac}\Big)^3t.
\end{align*}
\end{proposition}

\begin{proof}
Using (\ref{dynkinformula}), the second order moments
\[
p_t=\E(\Lambda_t\Gamma_t),\quad q_t=\E(\Lambda_t^2),\quad
r_t=\E(\Gamma_t^2), \quad u_t=\E(\Lambda_tN_t),\quad v_t=\E(\Gamma_t
N_t),\quad w_t=\E(N_t^2)
\]
satisfy the system of equations
\begin{align*}
\left[\begin{array}{c}
           p_t'\\
           q_t'\\
           r_t'\\
           u_t' \\
           v_t'
\end{array}\right]
+ &\left[\begin{array}{ccccc}
         b+c+d & -1 & -ac & 0 & 0 \\
         -2ac & 2b & 0 & 0 & 0    \\
          -2 & 0 & 2(c+d) & 0 & 0 \\
          -c & 0 & 0 & b & -ac    \\
          0 & 0 & -c & -1 & c+d
\end{array}\right]
\left[\begin{array}{c}
           p_t\\
           q_t\\
           r_t\\
           u_t \\
           v_t
\end{array}\right]\\
&=\left[\begin{array}{c}
           b\lambda_0-ac\\
           ca^2\\
           c+d\\
           ac \\
           -c
\end{array}\right]g_t
  + \left[\begin{array}{c}
           0\\
           2b\lambda_0 \\
           1\\
           0 \\
           0
\end{array}\right]\ell_t
  + \left[\begin{array}{c}
           0\\
           0\\
           0\\
           b\lambda_0 \\
           0
\end{array}\right]m_t
\end{align*}
and, furthermore, $w'_t=2cv_t+cg_t$.  Rewriting, the functions
\[
\bar p_t=C(\Lambda_t,\Gamma_t)=p_t-g_t\ell_t,\quad \bar q_t=V(\Lambda_t)=q_t-\ell_t^2,\quad
\bar r_t=V(\Gamma_t)=r_t-g_t^2,
\]
solve
\begin{align*}
\left[\begin{array}{c}
          \bar p_t'\\
           \bar q_t'\\
           \bar r_t'
\end{array}\right]
+ &\left[\begin{array}{ccc}
         b+c+d & -1 & -ac  \\
         -2ac & 2b & 0  \\
          -2 & 0 & 2(c+d)
\end{array}\right]
\left[\begin{array}{c}
           \bar p_t\\
           \bar q_t\\
           \bar r_t
\end{array}\right]
=\left[\begin{array}{c}
           -ac\\
           ca^2\\
           c+d
\end{array}\right]g_t
  + \left[\begin{array}{c}
           0\\
           0\\
           1
\end{array}\right]\ell_t
\end{align*}
while $\bar u_t=C(\Lambda_t,N_t)$, $\bar v_t=C(\Gamma_t,N_t)$ and
$\bar w_t=V(N_t)$ are the solutions of
\begin{align*}
\left[\begin{array}{c}
          \bar u_t'\\
           \bar v_t'
\end{array}\right]
+ &\left[\begin{array}{cc}
         b & -ac  \\
         -1 & c+d
\end{array}\right]
\left[\begin{array}{c}
           \bar u_t\\
           \bar v_t
\end{array}\right]
=\left[\begin{array}{c}
           ac\\
           -c
\end{array}\right]g_t
  + c\left[\begin{array}{c}
           \bar p_t\\
           \bar r_t
\end{array}\right]
\end{align*}
and $\bar w_t'=2c \bar v_t+cg_t$. In the limit $t\to\infty$,
\[
\bar p_\infty=\frac{ca^2(c+d)g_\infty}{2(b+c+d)(b(c+d)-ac)},\quad
\bar q_\infty =\frac{ca^2((c+d)(b+c+d)-ac)g_\infty}{2(b+c+d)(b(c+d)-ac)},
\]
and
\[
\bar r_\infty=g_\infty+\frac{ca^2 g_\infty}{2(b+c+d)(b(c+d)-ac)}.
\]
Moreover,
\begin{align*}
\left[\begin{array}{c}
          \bar u_\infty\\
           \bar v_\infty
\end{array}\right]
&=\frac{1}{b(c+d)-ac}\left[\begin{array}{cc}
         c+d & ac  \\
         1 & b
\end{array}\right]
\left[\begin{array}{c}
           acg_\infty +c\bar p_\infty\\
           c(\bar r_\infty-g_\infty)
\end{array}\right]\\
&=\left[\begin{array}{c}
           c+d+\frac{ac((c+d)^2+ac)}{2(b+c+d)(b(c+d)-ac)}\\
           1+\frac{ac}{2(b(c+d)-ac)}
\end{array}\right] \frac{acg_\infty}{b(c+d)-ac}.
\end{align*}
Finally,
\begin{align*}
V(N_t)\sim c(2\bar v_\infty+g_\infty)t
   &=\frac{\lambda_0 c}{c+d}\Big(\frac{b(c+d)}{b(c+d)-ac}\Big)^3t.
\end{align*}
\end{proof}

\section{Branching process representation}

The market orders component $N_t$ in the buffer-Hawkes process admits
a branching process representation analogous to that of the original
Hawkes process as in (\ref{branchingclassical}). The branching
representation emphasizes the clustering in time, which occurs as a
result of self-excitation, or mutual excitation as modelled above. It
also helps to give a straight proof of the diffusion limit for finite
dimensional distributions.

To derive the branching representation we use the same probabilistic
setting as detailed in the construction of the Markov process $(X_t)$
but now suppose that $\Nc(ds,du)$ and $\Mc(ds,du)$ are Poisson
measures on $\R_+\times\R_+$ with intensity measures $n(ds,du)$ and
$m(ds,du)$, respectively, given by
\begin{equation}\label{Poissonint}
n(ds,du)=\lambda_0ds\,ce^{-(c+d)u}du,\quad
m(ds,du)=a e^{-bs}ds\,ce^{-(c+d)u}du.
\end{equation}
To see why these intensities reflect the dynamics of $X$, consider a
time $t$ when the size of the LOB is some number
$\Gamma_t=\gamma$. Then the rate of execution is $c\gamma$ and the
rate of cancellation is $d\gamma$. In other words, each single limit
order has execution rate $c$ and cancellation rate $d$ and hence
remains in the book during an exponentially distributed time with
total intensity $c+d$. After independent thinning with probability
$c/(c+d)$ one obtains the intensity $n(ds,du)$ to have among the
regular arrivals of limit orders at rate $\lambda_0$ an entry at time
$s$ which is executed at time $s+u$.  Similarly, $m(ds,du)$ is the
Poisson intensity for the corresponding event to occur as the result
of a separate contribution $ae^{-bs}$, $s\ge 0$, to the arrival
intensity of limit orders.  Consequently,
\[
N_t^{(0)}= \int_{\R_+\times\R_+} 1_{\{s+u\le t\}}\,\Nc(ds,du)
        = \int_0^t\int _0^{t-s} \,\Nc(ds,du)
\]
is the number of market orders in $[0,t]$ generated by the base level
intensity $\lambda_0$. At the time of execution each of these
events independently adds to the intensity of limit orders. Due to the
Hawkes mechanism the new intensity contributions are governed by the
Poisson measure $\Mc$ and come in the shape of shot-noise profiles of
height $a$ and removal rate $b$. Each new market order executed as a
result of the added intensity repeats the procedure independently, and
therefore the total number of market orders grow as the total number
of progeny in a branching process.  Let $Z_t$, $t\ge 0$, denote the
total number of individuals at time $t$ in a continuous time branching
process with $Z_0=1$ and offspring distribution on $\R_+$ consisting
of one offspring unit at birth time $s+u$ for each Poisson point
$(s,u)$ of $\Mc$. Because of (\ref{stabilitycondition}) the branching
process is subcritical with mean offspring
\begin{equation} \label{nu}
\nu=\int_{\R_+\times \R_+} m(ds,du)=\int_0^\infty\int _0^\infty ae^{-bs}
ce^{-(c+d)u}\,dsdu =\frac{ac}{b(c+d)}<1.
\end{equation}
Let $Z^{(s,u)}$, $(s,u)\in \Mc$, denote a collection of independent
copies of $Z$. The branching property relation is
\begin{equation}\label{branching}
Z_t\stackrel{d}{=}
1+\int_0^t\int _0^{t-s} Z^{(s,u)}_{t-s-u}\,\Mc(ds,du).
\end{equation}
In conclusion, the combined number of market orders arise as
\begin{align}\label{marketorders}
 N_t&=\int_0^t\int _0^{t-s} Z^{(s,u)}_{t-s-u}\,\Nc(ds,du),
\end{align}
where the $Z$-processes are independent of $\Nc$. Thus, the component
$(N_t)$ of $X$ is also a
branching process with immigration, such that single immigrants arrive
at times $s+u$ for each Poisson point $(s,u)$ of $\Nc$ and each
immigrant generate independent offspring according to $Z$.  It is
sometimes convenient to split up the
immigrants counted by $N^{(0)}_t$ from the further jumps of $N_t$
along trajectories of $Z_t$ due to the Hawkes feedback, that is
\begin{align}\label{marketorderssplit}
 N_t 
&=N_t^{(0)}+\int_0^t\int _0^{t-s} (Z_{t-s-u}^{(s,u)}-1)\,\Nc(ds,du).
\end{align}

\subsection{Properties of the subcritical branching process}

Since the underlying branching mechanism is subcritical, as
$t\to\infty$ the almost surely nondecreasing process $Z_t$ reaches a
limit $Z_\infty<\infty$, $\P$-a.s., attaining the distribution of the
ultimate total progeny in a Galton Watson process with Poisson$(\nu)$
offspring distribution starting from a single individual.
It is well-known that $Z_\infty$ has a Borel distribution with mean
$(1-\nu)^{-1}<\infty$.  These facts are recalled next as we derive
bounds for the the moment generating functions $\E[e^{\theta Z_t}]$
and $\E[e^{\theta N_t}]$. By (\ref{branching}), for all $\theta\le 0$
at least,
\[
\ln \E[e^{\theta Z_t}]=\theta +\int_0^t \int_0^{t-s} \big(\E[e^{\theta
  Z_{t-s-u}}]-1\big)\,m(ds,du)
\]
Let
\[
V_t(\theta)=\ln \E[e^{\theta Z_t}],\quad
V_\infty(\theta)=\ln \E[e^{\theta Z_\infty}],
\]
for each $\theta$ where these functions exist finitely. The above integral
equation,
\[
V_t(\theta)=\theta+\int_0^t \int_0^s
\big(e^{V_{s-u}(\theta)}-1\big)\,ae^{-b(t-s)} ds\,c^{-(c+d)u}du,
\]
shows that $V_t(\theta)$ is differentiable in $t$ and solves the
nonlinear ODE
\begin{equation}\label{logmgfode}
V''_t(\theta)+(b+c+d)V'_t(\theta)+b(c+d)V_t(\theta)
  =b(c+d)\theta+ac(e^{V_t(\theta)}-1),
\quad V_0(\theta)=\theta,\; V'_0(\theta)=0.
\end{equation}
By monotone convergence,
\[
V_\infty(\theta)= \theta+\nu \big(e^{V_\infty(\theta)}-1\big).
\]
Equivalently
\[
-(V_\infty(\theta)-\theta+\nu)\,\exp\{-(V_\infty(\theta)-\theta+\nu)\}=-\nu
e^{\theta-\nu}.
\]
For each $\theta$ such that $-e^{-1}\le -\nu e^{\theta-\nu}<0$,
this equation has two solutions given by the two branches $W_0$ and
$W_{-1}$ of the real valued Lambert-$W$ function.  The secondary
branch $W_{-1}$ is excluded since the property $W_{-1}(x)\le -1$,
$e^{-1}\le x<0$, would imply $V_\infty(\theta)>\theta$ for all $\theta$, which is
not the case.  The relevant solution in terms of the
primary branch $W_0$ of the Lambert-$W$ function is
\[
V_\infty(\theta)=\theta-\nu-W_0(-\nu e^{\theta-\nu}),\quad
\theta<\theta_0=-\ln \nu +\nu-1.
\]
Since $\theta_0>0$ for $0<\nu<1$ and $W_0(e^{-1})=-1$,
the (logarithmic) moment generating functions $V_t(\theta)$ and
$V_\infty(\theta)$ exist finitely for $\theta$ in an open interval
containing $0$, namely
\begin{equation}\label{finitecumulant}
V_t(\theta)\le V_\infty(\theta)\le V_\infty(\theta_0)=-\ln \nu,\quad
\theta\le \theta_0.
\end{equation}
In particular, the moments of any order $n$ are
finite,
\begin{equation}\label{finitemoment}
\E Z_t^n\le \E Z_\infty^n<\infty,\quad t>0.
\end{equation}
The defining property $W(x)e^{W(x)}=x$ of the Lambert-W function
implies
\[
\E[e^{\theta Z_\infty}]=e^{\theta-\nu}e^{-W_0(-\nu e^{-\nu}e^\theta)}
   = \frac{W_0(-\nu e^{-\nu}e^\theta)}{-\nu},\quad \theta\le \theta_0.
\]
An application of the Taylor series of $W_0$ around $0$,
\[
W_0(x)=\sum_{n=1}^\infty \frac{(-1)^{n-1}}{n!} x^n,\quad |x|\le e^{-1},
\]
reveals the Borel distribution
\[
\P(Z_\infty=k)=\frac{(k\nu)^{k-1}}{k!}e^{-\nu k},\quad k=1,2\dots
\]
It is seen from (\ref{logmgfode}) that the mean and variance functions
\[
x_t=\E Z_t=\frac{d}{d\theta}V_t(\theta)\big|_{\theta=0},\quad
y_t=\mathrm{Var} Z_t=\frac{d^2}{d\theta^2}V_t(\theta)\big|_{\theta=0},
\]
satisfy
\[
x_t''+(b+c+d)x_t'+(b(c+d)-ac)x_t=b(c+d),\quad x_0=1,\; x'_0=0
\]
and
\[
y_t''+(b+c+d)y_t'+(b(c+d)-ac)y_t=ac x_t^2,\quad y_0=0,\; y'_0=0.
\]
The solutions are
\[
x_t
 =1+ac\, \int_0^t \frac{e^{-q_- s}-e^{-q_+ s}}{q_+-q_-}\,ds
\]
and
\[
y_t= ac \int_0^t x_s^2 \,\frac{e^{-g_-(t-s)}-e^{-g_+(t-s)}}{q_+-q_-}\,ds
\]
Moreover, as $t\to\infty$,
\begin{align*}
x_t &\to x_\infty=\E Z_\infty=\frac{b(c+d)}{b(c+d)-ac}=\frac{1}{1-\nu},\\
y_t &\to y_\infty=\mathrm{Var}Z_\infty= (x_\infty-1)x_\infty^2=\frac{\nu}{(1-\nu)^3}.
\end{align*}
As for properties of the increments of $Z_t$ we mention the following
property which will be used to obtain the covariance of a stationary
increments version of the buffer-Hawkes process.

\begin{proposition} \label{prop_integralw}
The squared increment expectation
\[
w_r(t)=\E[(Z_{r+t}-Z_r)^2], \quad t\ge 0,
\]
is integrable over $r$, such that
\begin{align*}
\int_0^\infty w_r(t)\,dr &=\frac{1}{1-\nu}
 \int_0^\infty (x_{r+t}-x_r)^2\,dr\\
&\quad + \frac{\nu}{1-\nu} \int_0^t
\frac{(c+d)e^{-b(t-u)}-be^{-(c+d)(t-u)}}{c+d-b}
    (x_u^2+y_u)\,du<\infty
\end{align*}
is uniformly bounded in $t$.
\end{proposition}

\begin{proof}
Using
\begin{align*}
Z_{r+t}-Z_r&=\int 1_{\{0<s+u<r\}}(Z^{(s,u)}_{r+t-s-u}-Z^{(s,u)}_{r-s-u})\,\Mc(ds.du)\\
 &\quad + \int 1_{\{r<s+u<r+t\}}Z^{(s,u)}_{r+t-s-u}\,\Mc(ds.du),
\end{align*}
it follows
\begin{align*}
w_r(t)&=\E[Z_{r+t}-Z_r]^2+\int 1_{\{0<s+u<r\}} \E[(Z_{r+t-s-u}-Z_{r-s-u})^2]\,m(ds.du)\\
 &\quad +  \int 1_{\{r<s+u<r+t\}}\E[Z_{r+t-s-u}^2]\,m(ds.du)\\
&= (x_{r+t}-x_r)^2+\int 1_{\{0<s+u<r\}} w_{r-s-u}(t)\,m(ds.du)\\
&\quad + \int 1_{\{r<s+u<r+t\}}(x_{r+t-s-u}^2+y_{r+t-s-u})\,m(ds.du).
\end{align*}
An integration of this relation over $r$ yields
\begin{align*}
\int_0^\infty w_r(t)\,dr &= \int_0^\infty (x_{r+t}-x_r)^2\,dr
+\frac{ac}{b(c+d)}\int_0^\infty w_r(t)\,dr\\
&\quad + \frac{ac}{b}
\int_0^t \frac{(c+d)e^{-b(t-u)}-be^{-(c+d)(t-u)}}{(c+d)(c+d-b)} (x_u^2+y_u)\,du,
\end{align*}
which yields the stated integral expression.  It is straightforward to
use the explicit representations for $x_t$ and $y_t$ to check the
uniform boundedness in $t$.
\end{proof}

\subsection{Further properties of $N_t$ derived from cluster representation}

Using (\ref{marketorders}),
\begin{align}\nonumber
\ln \E[e^{\theta N_t}]&=
\int_0^t \int_0^{t-s} (e^{V_{t-s-u}(\theta)}-1)\lambda_0
ce^{-(c+d)u}\,duds\\\label{logmgfN}
&=\frac{\lambda_0 c}{c+d} \int_0^t (e^{V_u(\theta)}-1)(1-e^{-(c+d)(t-u)})\,du,
\end{align}
so by (\ref{finitecumulant})
\[
 \ln \E[e^{\theta N_t}]\le
\frac{\lambda_0 c}{(c+d)^2}\frac{1-\nu}{\nu}\,((c+d)t-1+e^{-(c+d)t}),
\quad \theta\le \theta_0.
\]
The mean $m_t=\E N_t$ derived in Proposition 2.1, is further related
to $x_t=\E Z_t$ via the ODE
\begin{align*}
m_t''+(c+d)m_t'=c(g_t'+(c+d)g_t)  
           =\lambda_0 c\, x_t.
\end{align*}
The solution of this equation with $m_0=m'_0=0$ is
\begin{align*}
\E N_t&=\int_0^t\int _0^{t-s} \E Z_{t-s-u}\,n(ds,du)\\
      &=\frac{\lambda_0c}{c+d}\int_0^t x_u(1-e^{-(c+d)(t-u)})\,du,
\end{align*}
which is also immediate from (\ref {logmgfN}) by differentiation with
respect to $\theta$.  The same method allows us to derive an expression
for the variance of $N_t$, which is more convenient than the previous
$\bar w_t$, namely
\begin{align*}
\mathrm{Var} N_t
                &= \int \E[Z_{t-s-u}^2]\,n(ds,du)\\
     &=\frac{\lambda_0c}{c+d}\int_0^t (y_u+x_u^2)\,(1-e^{-(c+d)(t-u)})\,du.
\end{align*}
Using a scaling parameter $m$, the asymptotic rate of growth of market
orders as $m\to\infty$ is
\[
\frac{1}{m} \E N_{mt}\to \frac{\lambda_0 c}{c+d} x_\infty t
= \frac{\lambda_0 c}{c+d}\frac{1}{1-\nu}\, t,
\]
with
\[
\frac{1}{m} \mathrm{Var} N_{mt}\to \frac{\lambda_0
  c}{c+d}(x_\infty^2+y_\infty) t
=\frac{\lambda_0 c}{c+d} \frac{1}{(1-\nu)^3} t,
\]
in agreement with Proposition 2.2.  Putting these relations together
we obtain the weak law of large numbers. On the other hand,  the strong law  given by
\[
\frac{N_{mt}}{m}\longrightarrow \frac{\lambda_0 c}{c+d}\frac{1}{1-\nu}\,
t,\quad \quad    m\to \infty,
\]
 follows from the ergodic theorem in view of the existence of a stationary increments version of $N$ as given in subsection \ref{stationary} below. Note that stationarity is based on the stability assumption $\nu<1$, and ergodicity is implied by that of the time shifts of a Poisson random measure.

\section{Diffusion Limit}

In this section, we consider the long-time scaling of the market
orders to obtain a diffusion limit. This component of buffer-Hawkes
process is in our focus because the price formation will be based on
the market orders in the sequel.

\subsection{Diffusion scaling of the market order process}

Let us introduce the centered and scaled market orders $N_t^{(m)}$ by
putting
\[
\bar N_t=N_t-\E N_t,\quad N^{(m)}_t=\frac{\bar N_{mt}}{\sqrt{m}},\quad
t\ge 0.
\]
We need the following tightness result to prove functional convergence of the scaled process.

\begin{lemma} \label{onlylem} The sequence of processes $\{N_\cdot^{(m)}\}_{m\ge 1}$ is tight.
\end{lemma}

\begin{proof} Let
\[
M_t=N_t-c\int_0^t \Gamma_s\,ds, \quad
A_t=c\int_0^t (\Gamma_s-g_s)\,ds,\quad
B_t=c^2\int_0^t \Gamma_s\,ds.
\]
Then $M_t$ is a $(\P,\Fc)$-martingale, $A_t$ is the drift and
$\langle M,M\rangle_t=B_t$ is the quadratic variation in the
semimartingale decomposition of $\bar N_t$, given by $\bar N_t=A_t+M_t$,
$t\ge 0$.   Let $(\tau_m)_{m\ge 1}$ be a family of $(\Fc)$-stopping times
all bounded by some constant $T$, $\sup_m \tau_m\le T$.
To verify the Aldous-Rebolledo tightness criterion (see e.g. \cite{etheridge}) we will show
that for each $\epsilon>0$ there exist $\delta>0$ and an integer $m_0$ such that
\begin{equation}\label{tightness}
\sup_{m\ge m_0}\sup_{h\in
  [0,\delta]}\P\Big(\Big|\frac{A_{m(\tau_m+h)}-A_{m\tau_m}}{\sqrt{m}}\Big|>
\epsilon\Big)\le \epsilon,
\end{equation}
and furthermore that the same boundedness property holds when $A_t$ is
replaced by $B_t$ in (\ref{tightness}).

 By Chebyshev's inequality,
\[
\P\Big(\big|
A_{m(\tau_m+h)}-A_{m\tau_m}\big|\ge \epsilon\sqrt{m}\Big) \le
\frac{c^2}{\epsilon^2 } \E\Big[\Big(\int_{\tau_m}^{\tau_m+h}
(\Gamma_{ms}-g_{ms})\,ds\Big)^2\Big]
\]
Without restricting the scope of the proof we may take $h\le 1$. Put
$T'=T+1$.  By Hölder's inequality,
\[
\Big(\int_{\tau_m}^{\tau_m+h} (\Gamma_{ms}-g_{ms})\,ds\Big)^2
\le \int_0^{T'}(\Gamma_{ms}-g_{ms})^2\,ds\, \cdot\int_{\tau_m}^{\tau_m+h}
\,ds
=\int_0^{T'}(\Gamma_{ms}-g_{ms})^2\,ds\, \cdot h,
\]
and combining the two previous bounds
\[
\P\Big(\big|
A_{m(\tau_m+h)}-A_{m\tau_m}\big|\ge \epsilon\sqrt{m}\Big)
 \le \frac{c^2}{\epsilon^2}
 \int_0^{T'} \mathrm{Var}(\Gamma_{ms})\,ds\, \cdot h.
\]
The function $\bar r_t=\mathrm{Var}(\Gamma_t)$ with $\bar r_\infty<\infty$
obtained in (), is bounded on the real line by $\bar r_\mathrm{sup}=
\sup_{t\ge 0}\mathrm{Var}(\Gamma_t) <\infty$.  Thus, for any $m$,
\[
\sup_{h\in [0,\delta]} \P\Big(\big|
A_{m(\tau_m+h)}-A_{m\tau_m}\big|\ge \epsilon\sqrt{m}\Big)
\le \frac{c^2}{\epsilon^2}\, T'\, \bar r_\mathrm{sup}\,\delta.
\]
Take $\delta=\epsilon^3/(c^2 (T+1) \bar r_\mathrm{sup})$ to obtain
(\ref{tightness}) for the drift process $A_t$.  The same arguments
using $r_\mathrm{sup}=\sup_{t\ge 0} \E[\Gamma_t^2]<\infty$ instead of
$\bar r_\mathrm{sup}$ shows that (\ref{tightness}) holds for the
quadratic variation process $B_t$.
\end{proof}

Now, we are ready to prove the diffusion limit of $N^{(m)}$ in the following theorem where asymptotic variance is found in terms of the parameters of the buffer-Hawkes process.
\begin{theorem}
$\{N^{(m)}_t\}_{t\ge 0}$ converges weakly as
$m\to\infty$ in the space $D([0,\infty),\R)$ of realvalued c\`adl\`ag
processes to Brownian motion with variance coefficient
\[
\sigma^2=\frac{\lambda_0c}{c+d}\frac{1}{(1-\nu)^3}.
\]
\end{theorem}
\begin{proof} We show convergence of the finite-dimensional distributions.
Put
\[
\Phi(x)=e^{x}-1-x.
\]
The cumulant function of $\bar N_t$,
\begin{align*}
\ln \E \exp\{\theta \bar N_t\}&=
  \int_{\R_+^2} 1_{\{s+u\le t\}} \E[\Phi(\theta Z_{t-s-u})]\,n(ds,du)\\
    &=   \int_{\R_+^2} 1_{\{s+u\le t\}}
    \big(e^{V_{t-s-u}(\theta)}-1-\theta x_{t-s-u}\big) \,n(ds,du),
\end{align*}
exists finitely for each $\theta\le \theta_0$, due to
(\ref{finitecumulant}), (\ref{finitemoment}).
More generally,
\[
\ln \E \exp\Big\{\sum_{i=1}^n \theta_i \bar N_{t_i}\Big\}=
  \int_{\R_+\times \R_+} \E\Big[\Phi\Big(\sum_{i=1}^n
  \theta_i Z_{t_i-s-u}\,1_{\{s+u\le t_i\}}\Big)\Big] \,n(ds,du)
\]
is well-defined for $0\le t_1\le \dots\le t_n$ and
$\theta_1,\dots,\theta_n$, $n\ge 1$, with $\sum_{k=1}^n\theta_k\le
\theta_0$.
Under scaling with scaling parameter
$m\to\infty$, Hölder's inequality with $p,q>1$,
$\frac{1}{p}+\frac{1}{q}=1$ applies to control the remainder term using
\[
m\E\Big[\Big|\Phi(\frac{X}{\sqrt{m}})-\frac{1}{2}\frac{X^2}{m}\Big|\Big]\le
  \frac{1}{\sqrt{m}} \,\E\big[|X|^3\,e^{|X|}\big]
  \le \frac{1}{\sqrt{m}} \,\E\big[|X|^{3p}\big]^{1/p}\,\E\big[e^{q|X|}\big]^{1/q},
\]
for a generic random variable $X$. Indeed, with
\[
0<\theta_0'<\theta_0, \quad  \sum_{k=1}^n\theta_k\le \theta_0'\quad q=
\frac{\theta_0}{\theta_0'}>1,
\]
and  using (\ref{finitemoment}), for large $m$,
\begin{align*}
\ln \E &\exp \Big\{i\sum_{i=1}^n \theta_i N^{(m)}_{t_i}\Big\}\\
&= -\frac{1}{2m}  \int_{\R_+\times \R_+} \E\Big[\Big(\sum_{i=1}^n
  \theta_i Z_{mt_i-s-u}\,1_{\{s+u\le mt_i\}}\Big)^2 \Big]
  \,n(ds,du)+O(\frac{1}{\sqrt{m}})\\
&\;\sim -\frac{1}{2}\sum_{1\le i,j\le n} \theta_i\theta_j
 \int_{\R_+\times \R_+} \E [Z_{m(t_i-s)-u} Z_{m(t_j-s)-u}]\,
1_{\{s+u/m\le (t_i\wedge t_j)\}}   \,n(ds,du).
\end{align*}
As  $m\to\infty$, the leading
term on the right hand side converges to
\[
 -\frac{1}{2}\sum_{1\le i,j\le n} \theta_i\theta_j
 \int_0^{t_i\wedge t_j}\int_0^\infty  \E(Z_\infty^2)
  \,n(ds,du)= -\frac{\sigma^2}{2}
\sum_{1\le i,j\le n} \theta_i\theta_j\, (t_i\wedge t_j)
\]
and hence
\begin{align*}
\ln \E \exp\Big\{i\sum_{i=1}^n \theta_i N^{(m)}_{t_i}\Big\}
 \to
  -\frac{\sigma^2}{2}
\sum_{1\le i,j\le n} \theta_i\theta_j\, (t_i\wedge t_j).
\end{align*}
In view of Lemma \ref{onlylem}, the proof is complete.
\end{proof}

\subsection{Price process and its long-time scaling}

Suppose $N^+$ and $N^-$ are copies of the above buffer-Hawkes process,
representing market call order and market bid order executions,
respectively.  These two counting processes are naturally associated with up
and down movements of the trading price of the underlying asset.
The simplest trading price is the midprice formed as in \cite[pg.74]{Zaatour} by
\[
S_t=\left(N^{+}_t-N^{-}_t\right)\, \frac{\alpha}{2},
\]
where $+$ stands for up, and $-$ stands for down, and $\alpha>0$ is a
tick price parameter, and we have taken $S_0=0$ for simplicity.

When $N^+$ and $N^-$ are independent from each other and identical in distribution, consider $S^{(m)}_t= {  S_{mt}}/{\sqrt{m}}$.  As a result of the diffusion limit for each $N^+$ and $N^-$, we have that
$\{S^{(m)}_t\}_{t\ge 0}$ converges weakly as
$m\to\infty$ in the space $D([0,\infty),\R)$ of real valued c\`adl\`ag
processes to Brownian motion with variance coefficient $\beta$, namely, the asymptotic volatility. For the proof of this fact, we observe the semi-martingale decomposition
\[
N^{+}_t-N^{-}_t=N^{+}_t-c\int_0^t \Gamma^+_s ds \; -\;
N^{-}_t+c\int_0^t \Gamma^-_s ds + A_t
\]
where  the drift is given by
\[
A_t = c\int_0^t \Gamma_s^{+}\, ds - c\int_0^t \Gamma_s^{-}\, ds,
\]
and the quadratic variation of the martingale part of $N^+-N^-$ is
\[
B_t:=c^2\int_0^t \Gamma_s^{+}\, ds + c^2\int_0^t \Gamma_s^{-}\, ds
\]
Therefore, the tightness proof is similar to that of $N^+$ or $N^-$. We get   the asymptotic volatility as
\[
\beta:=\frac{\alpha^2\sigma^{2}}{2}   = \frac{\alpha^2\lambda_0 c}{2(c+d)(1-\nu)^3}
\]
where $\nu=\frac{ac}{b(c+d)}<1$ by \eqref{nu}. Note that the
volatility increases as $\nu$ increases due to an increase in either
$a/b$ or $c/d$ in the buffer-Hawkes model, whereas the volatility
solely depends on $a/b$ in the Markovian Hawkes process
\cite[pg.75]{Zaatour}.

\subsection{Stationary increments version of $N_t$ and a covariance
  formula.} \label{stationary}

The model considered so far starts with an empty limit order book at
time zero, $\Gamma_0=0$. The transition of $\Gamma_t$ during a
start-up phase from its initial value to approaching the
steady-state $\Gamma_\infty$ imposes a non-stationary behavior to $N$.
In this section, we construct a version of $N_t$ with stationary
increments. Of course, the diffusion approximation will be the same as
for the original process.  Apart from achieving linear increase in $t$
of the expected value, the main advantage is an explicit
expression for the covariance along the paths of the $N$ process.

Recall that in the original model each market order occurs at a time
$s+u$, where $s>0$ is the entry time of a limit order in the LOB and
$u$ is the occupation time in the book until the order is
executed. The process $N_t$ counts all orders with $s+u\le t$,
including $N_t^{(0)}$ as well as additional market orders accounted
for by $Z_{t-s-u}^{(s,u)}-1$ in (\ref{marketorderssplit}).  To obtain a
version $\widetilde N_t$ of $N_t$ with stationary increments we extend
the construction of the buffer-Hawkes process $X$ using a Poisson
measure $M$ defined on $\R\times \R_+$ in (\ref{Npoissondef}), make the
corresponding adjustments elsewhere whenever needed, and move the
initiation time of the limit orders backwards to include orders placed
at times $s<0$. For such an $s$, if $s+u\le 0$ then the resulting
number of market orders in $[0,t]$ is given by
$Z^{(s,u)}_{t-s-u}-Z^{(s,u)}_{0-s-u}$. If $0<s+u\le t$ then as before
the contribution to $\widetilde N_t$ is $Z^{(s,u)}_{t-s-u}$. Upon
summing over all Poisson points in $\Nc$ we define
\[
\widetilde N_t=\int_{\R\times \R_+} \big\{ 1_{\{s+u\le 0\}}
(Z_{t-s-u}^{(s,u)}-Z^{(s,u)}_{0-s-u})+1_{\{0<s+u\le t\}}
Z_{t-s-u}^{(s,u)}\big\}\,\Nc(ds,du)
,\quad t\ge 0.
\]
Here, re-ordering terms,
\begin{align*}
\widetilde N_{r+t}- \widetilde N_r
&=\int_{\R\times \R_+} \big\{ 1_{\{s+u\le 0\}}(Z_{r+t-s-u}^{(s,u)}-Z^{(s,u)}_{0-s-u})
+1_{\{0<s+u\le r+t\}} Z_{r+t-s-u}^{(s,u)}\big\}\,\Nc(ds,du)\\
&\quad -\int_{\R\times \R_+}
\big\{ 1_{\{s+u\le 0\}}(Z_{r-s-u}^{(s,u)}-Z^{(s,u)}_{0-s-u})+1_{\{0<s+u\le r\}}
  Z_{r-s-u}^{(s,u)}\big\}\,\Nc(ds,du)\\
&=\int_{\R\times \R_+} \big\{ 1_{\{s+u\le r\}}(Z_{r+t-s-u}^{(s,u)}-Z^{(s,u)}_{r-s-u})
+1_{\{r<s+u\le r+t\}} Z_{r+t-s-u}^{(s,u)}\big\}\,\Nc(ds,du)
\end{align*}
and hence
by the
shift-invariance of $\Nc(ds,du)$ with respect to $s$ we obtain the
desired stationary increments property
\[
\widetilde N_{r+t}- \widetilde N_r\stackrel{d}{=}\widetilde N_t.
\]

The stationary mean value is obtained as
\begin{align*}
\widetilde m_t&=\E(\widetilde N_t)=m_t+
\int_{\R\times \R_+}  1_{\{s\le 0,0\le s+u\le t\}}
x_{t-s-u}\,n(ds,du)\\
&\quad +\int_{\R\times \R_+}  1_{\{s+u\le 0\}}
(x_{t-s-u}-x_{0-s-u})\,n(ds,du)\\
&=\frac{\lambda_0 c}{c+d}\int_0^t x_s\,ds
   +\frac{\lambda_0 c}{c+d}\frac{ac}{q_+-q_-}
\int_0^t(\frac{e^{-q_-v}}{q_-}-\frac{e^{-q_+v}}{q_+}) \,dv\\
&=\dots =\frac{\lambda_0c bt}{b(c+d)-ac}
\end{align*}
and the variance is
\begin{align*}
\mathrm{Var}\widetilde N_t&=
\int_{\R\times \R_+}  1_{\{0\le s+u\le t\}}
\E[Z_{t-s-u}^2]\,n(ds,du)\\
&\quad +\int_{\R\times \R_+}  1_{\{s+u\le 0\}}
\E[(Z_{t-s-u}-Z_{0-s-u})^2]\,n(ds,du)\\
&=\frac{\lambda_0 c}{c+d}\int_0^t \E[Z_s^2] \,ds
   +\frac{\lambda_0 c}{c+d} \int_0^\infty \E[(Z_{r+t}-Z_r)^2]\,dr.
\end{align*}
By Proposition \ref{prop_integralw}, the second term in
\begin{align*}
\mathrm{Var}\widetilde N_t
&=\frac{\lambda_0 c}{c+d}\int_0^t (x_s^2+y_s)\,ds
   +\frac{\lambda_0 c}{c+d} \int_0^\infty w_r(t)\,dr
\end{align*}
vanishes in the scaling limit, and
\begin{align*}
\frac{1}{m}\mathrm{Var}\widetilde N_{mt}
&\to \frac{\lambda_0 c}{c+d}(x_\infty^2+y_\infty)\,t
=\frac{\lambda_0 c}{c+d}x_\infty^3\,t.
\end{align*}
Since the increments are stationary, for $s\le t$,
\begin{align*}
\mathrm{Cov}(\widetilde N_s,\widetilde N_t-\widetilde N_s)
&=\frac{1}{2}(\mathrm{Var}\widetilde N_t-\mathrm{Var}\widetilde
N_s-\mathrm{Var}\widetilde N_{t-s})\\
&=\frac{\lambda_0 c}{2(c+d)}\Big(
\int_s^t(x^2_u-x^2_{u-s})\,du+\int_s^t(y_u-y_{u-s})\,du\\
&\qquad\qquad  +\int_0^\infty (w_r(t)-w_r(s)-w_r(t-s))\,dr\Big).
\end{align*}
Clearly, $\mathrm{Cov}(\widetilde N_{ms},\widetilde N_{mt}-\widetilde
  N_ {ms})/m\to 0$, $m\to\infty$, in agreement with the Brownian
  motion scaling limit.

\section{Conclusions and Outlook}

With the emphasis in our model on self-exciting features and mutual
excitation of market events, it is natural to look further at the
implications of these mechanisms for applications in financial
mathematics and quantitative finance computations. To indicate
potential use of the model developed in this paper, we conclude by
mentioning a few examples.

\subsection*{Parameter estimation}

Suppose we extract trading data for a single asset and tentatively
apply the Markov model $(X_t)$ to capture its evolution over time.
The parameter ratio $c/(c+d)$ corresponds to the ratio of executed
versus cancelled limit orders, which is a directly observable
quantity. Also available is the average number of registered limit
orders, represented by $\E
\Gamma_\infty=\lambda_0(c+d)^{-1}(1-\nu)^{-1}$. The coefficient of
variation, that is, the ratio of sample variance to sample mean, of
many observed $N_{t+1}-N_t$, say, would moreover give a point estimate
of $\nu$, and hence of $a/b$.  It would be a separate investigation to
look into efficient and reliable estimation procedures set up along
these lines.

\subsection*{Price formation process}

During trading, at the time $s$ of a market order execution there is a
supply of $\Gamma_s$ limit orders in the LOB, each equipped with a bid
or call price. The market price change, up or down, is determined by
the most favorable offer in the LOB, which we may think of as a
minimum of the current limit order entries.  It is reasonable,
therefore, that the change in price will be inversely proportional to
$\Gamma_s$.  Thus, a more elaborate attempt to model the link between the
number of market orders and the trading price could lead to price
processes of the type
\[
S_t=\Big(\int_0^t \frac{1}{\Gamma^+_{s-}}\,1_{\{\Gamma^+_{s-}\ge 1\}}\,dN^+_s
    -\int_0^t \frac{1}{\Gamma^-_{s-}}\,1_{\{\Gamma^-_{s-}\ge 1\}}\,dN^-_s\Big)\frac{\alpha}{2}.
\]
The moments and long time behavior of the price can be studied under this model.

\subsection*{Geometric buffer-Hawkes process}

In parallel with geometric Poisson processes, see e.g. \cite[Ch.\
11]{shreve2004},  consider a price process of the form
\[
S_t=S_0 (1+\sigma)^{N_t} e^{\alpha t-\sigma \int_0^t c\Gamma_s\,ds},
\]
where $\sigma>-1$ is a constant. Then
$S_s-S_{s-}=\sigma S_{s-}\,dN_s$, and hence the discounted price
process $\bar S_t=e^{-\alpha t} S_t$ satisfies
\begin{align*}
\bar S_t&=\bar S_0 -\sigma c\int_0^t \bar S_s\Gamma_s\,ds
+\sigma \int_0^t \bar S_{s-}\,dN_s\\
&= S_0+\sigma \int_0^t \bar S_{s-}\,(dN_s-c \Gamma_s\,ds)
\end{align*}
which shows that $\bar S_t$ is a martingale with $\E \bar S_t=S_0$ and
mean return $\alpha$.

\subsection*{Non-Markovian extensions}

While our construction concerned the Markov process $(X_t)$, of course
the point process component $(N_t)$ is non-Markovian in its own
sake. The branching representation of $N$ directly provides further
extensions in analogy to those in (\ref{hawkes}). For example,
power-law kernels are used in current studies of high-frequency
trading data \cite{zhang2016}.  In
(\ref{branching}), consider a non-negative function $\phi$ on $\R_+$, let
$\Mc$ be a Poisson measure with intensity process
$m(ds,du)=\phi(s)\,c^ {-(c+d)u}\,du$, and replace (\ref{nu}) by the
condition
\[
\nu=\frac{c}{c+d}\int_0^\infty \phi(s)\,ds <1.
\]
Now, let $Z$ be the subcritical branching process related to $\Mc$ via
(\ref{branching}) and, as before, generate $N$ as in (\ref{marketorders}).
Some of the properties of $N$ we have studied have direct
counterparts in this more general situation, some would need to be
studied in greater detail.

\subsection*{Multi-variate extensions}

Multidimensional versions can be considered for modeling several
assets as in e.g.\ \cite{embrechtsetal2011,AberJedi2015}.  Here, we
mention the two-dimensional case with two buffer-Hawkes processes for
buying and selling orders, $X^+_t=(\Lambda^+_t,\Gamma_t^+, N^+_t)$ and
$X^-_t=(\Lambda^-_t,\Gamma^-_t, N^-_t)$, respectively.  The component
processes are defined as previously except that the execution of buy
orders trigger the arrival of sell orders, and vice versa, which
suggests the semimartingale relations
\[
\left\{
\begin{array}{l}
d\Lambda^+_t= b(\lambda_0-\Lambda^+_t)\,dt+ac\Gamma^-_t\,dt
+d\widetilde X^+_t[f_1]\\
d\Gamma^+_t= (\Lambda^+_t-(c+d)\Gamma^+_t)\,dt+d\widetilde X^+_t[f_2]
\\
dN^+_t= c\Gamma^+_t\,dt+d\widetilde X^+_t[f_3]
\end{array}
\right.
\]
and
\[
\left\{
\begin{array}{l}
d\Lambda^-_t= b(\lambda_0-\Lambda^-_t)\,dt+ac\Gamma^+_t\,dt
+d\widetilde X^-_t[f_1]\\
d\Gamma^-_t= (\Lambda^-_t-(c+d)\Gamma^-_t)\,dt+d\widetilde X^-_t[f_2]
\\
dN^-_t= c\Gamma^-_t\,dt+d\widetilde X^-_t[f_3],
\end{array}
\right.
\]
with $(\widetilde X^+[\cdot],\widetilde X^-[\cdot])$ a six-component
martingale, compare section \ref{sec:2.2}.


\end{document}